\documentclass[reqno]{amsart}

\usepackage{amsmath}
\usepackage{amssymb}
\usepackage{amsthm}
\usepackage{enumerate}
\usepackage[mathscr]{eucal}
\usepackage{eqlist}

\theoremstyle{plain}
\newtheorem{theorem}{Theorem}
\newtheorem{prop}{Proposition}
\newtheorem{lemma}{Lemma}

\theoremstyle{definition}

\theoremstyle{remark}

\numberwithin{equation}{section}

\renewcommand{\d}{\mathrm{d}}

\newcommand{\ZZ}{\mathbb{Z}}

\newcommand{\RR}{\mathbb{R}}

  \def\d{\,\mathrm{d}}
\def\div{\,|\,}  

\def\abs#1{\left|#1\right|} \def\Res{\mathop{\rm Res}\limits}
\def\h{\textstyle{\frac{1}{2}}}

 \def\be{\beta} \def\ep{\epsilon}
\def\br#1{\left(#1\right)}
\def\url#1{{\tt #1}} \def\be{\begin{equation}} \def\ee{\end{equation}}
\def\ba{\begin{aligned}} \def\ea{\end{aligned}}
\def\r{\mathcal{R}}   \def\A{\mathcal{A}} \def\U{\mathcal{U}}   \def\I{\mathcal{I}}

\def\Int{\int\limits} \def\inv#1{\frac{1}{#1}}

\begin{document}

\title[]%
{On a question of A.~Schinzel: Omega estimates for a special type of
arithmetic functions}
\author[]%
{Manfred K\"{U}HLEITNER and Werner Georg NOWAK}

\newcommand{\acr}{\newline\indent}

\address{\llap{\,}Institute of Mathematics\acr Department of Integrative
Biology\acr BOKU Wien\acr 1180 Vienna, AUSTRIA}
\email{kleitner@boku.ac.at,\quad nowak@boku.ac.at}

\thanks{The authors gratefully
acknowledge support from the Austrian Science Fund (FWF) under
project Nr.~P20847-N18.}


\begin{abstract} The paper deals with lower bounds for the remainder term in
asymptotics for a certain class of arithmetic functions. Typically,
these are generated by a Dirichlet series of the form
$\zeta^2(s)\zeta(2s-1)(\zeta(2s))^M\,H(s)$ where $M$ is an arbitrary
integer and $H(s)$ has an Euler product which converges absolutely
for $\Re s>\inv{3}$.
\end{abstract}


\maketitle

\centerline{\textit{To Professor Andrzej Schinzel on his 75th
birthday}}

\vskip1cm 

\section{Introduction} \subsection{} This article is concerned with a special
class $\frak C$ of arithmetic functions $f_H$ with a generating
Dirichlet series\footnote{The dependance of $f_H$ on the integer $M$
will be suppressed in notation throughout.} \be\label{basic}
F_H(s)=\sum_{n=1}^\infty \frac{f_H(n)}{n^s} =
\zeta^2(s)\zeta(2s-1)(\zeta(2s))^M\,H(s)\qquad(\Re s>1)\,, \ee where
$M$ is an integer, and $H(s)$ has an Euler product which converges
absolutely for $\Re s>\inv{3}$.

We mention some examples of special arithmetic interest: Firstly,
the function \begin{equation}\label{f_star}
    f^*(n):=\sum_{m\div n} \gcd\br{m,\frac{n}{m}}
\end{equation} (see
N.~Sloane \cite{Sloane_A055155}) in a way quantifies the property of
$n$ to be not square-free, i.e., to possess non-unitary divisors.
(For $n$ square-free, $f^*(n)$ coincides with the number-of-divisors
function $d(n)$.)  $f^*(n)$ is generated by the Dirichlet series
\begin{equation}
\sum_{n=1}^\infty \frac{f^*(n)}{n^s} =
\frac{\zeta^2(s)\zeta(2s-1)}{\zeta(2s)}\qquad\br{\Re(s)>1}\,.
\end{equation} This is \eqref{basic} with $M=-1$, $H(s)=1$ identically.

Secondly, consider \begin{equation}\label{}
    f_1(n)=\sum_{m\div n}\sigma\br{\gcd\br{m,\frac{n}{m}}}\,,
\end{equation}where $\sigma$ denotes the sum-of-divisors function:
cf.~again N.~Sloane \cite{Sloane_A124316}. The corresponding
generating function simply reads \be \label{F_one} F_1(s) =
\sum_{n=1}^\infty \frac{f_1(n)}{n^s} = \zeta^2(s)\zeta(2s-1)
\qquad(\Re s>1)\,.\ee This is \eqref{basic} with $M=0$, $H(s)=1$
identically.

As a third example, we mention the {\it modified Pillai's function}
(N.~Sloane \cite{Sloane_A078430})
\begin{equation}\label{p_star}
    P^*(n):=\frac{1}{n}\sum_{k=1}^n \gcd(k^2,n)\,.
\end{equation} This is generated by \be \sum_{n=1}^\infty \frac{P^*(n)}{n^s} =
\frac{\zeta^2(s)\zeta(2s-1)}{(\zeta(2s))^2}\,H^*(s)\qquad(\Re
s>1)\,, \ee where $H^*(s)$ has an Euler product absolutely
convergent for $\Re s>\inv{3}$.

\subsection{} The class of functions $\frak C$ has been dealt with in detail in a
recent paper by E.~Kr\"{a}tzel, W.G.~Nowak, and L.~T\'oth \cite{EGL}. In
that article, the emphasis was on upper bounds for the error term
$\r_{f_H}(x)$ in the asymptotic formula \be\label{Ass_fH} \sum_{n\le
x} f_H(n) = \Res_{s=1}\br{F_H(s)\frac{x^s}{s}}+ \r_{f_H}(x)\,. \ee
Since $F_H(s)$ has a triple pole at $s=1$, explicitly
$$\Res_{s=1}\br{F_H(s)\frac{x^s}{s}}=x\,p_H(\log x)\,,$$ where $p_H$ is
a quadratic polynomial whose coefficients depend on $H$ and $M$.
Using contour integration and properties of the Riemann
zeta-function, it has been proved in \cite{EGL} that \be \r_{f_H}(x)
= O\left(x^{2/3}(\log x)^{16/9}\right)\,. \ee Employing Kr\"{a}tzel's
method \cite{EKLP}, which involves fractional part sums and the
theory of (classic) exponent pairs, the slight refinement \be
\r_{f_H}(x) = O\left(x^{925/1392}\right)\ee has been obtained
($\frac{925}{1392}=0.6645\dots$). Finally, bringing in Martin
Huxley's deep and new technique \cite{Huxbook}, \cite{HuxLP},
\cite{HuxEZ} ("Discrete Hardy-Littlewood method"), the authors of
\cite{EGL} deduced the further improvement \be \label{hux}
\r_{f_H}(x) = O\left(x^{547/832}(\log x)^{26947/8320}\right)\qquad
\br{\textstyle\frac{547}{832}=0.65745\dots}\,. \ee For the context
of the class $\frak C$ within the frame of the theory of arithmetic
functions, as well as for a wealth of enlightening related results,
see also the recent papers by L.~T\'oth \cite{TothTorino},
\cite{TothInt}.

\subsection{} The results of \cite{EGL} have been presented at the
$20^{\rm th}$ Czech and Slovak International Conference on Number
Theory in Star\'a Lesn\'a, September 2011 \cite{EGLStara}. At the
end of that talk, Professor Andrzej Schinzel raised the following
question: "What can be said concerning Omega-estimates for the
remainder term?"

The authors of the present paper are very grateful for this valuable
stimulation of further research and are pleased to be able to
provide the following answer.

\begin{theorem} For any arithmetic function $f_H\in{\frak C}$ with a
generating function $F_H(s)$ according to \eqref{basic}, it holds
true that, as $x\to\infty$, $$ \sum_{n\le x} f_H(n) =
\Res_{s=1}\br{F_H(s)\frac{x^s}{s}}+ \Omega\br{\frac{\sqrt{x}(\log
x)^2}{(\log\log x)^{|M+1|}}}\,. $$ \end{theorem}

\subsection{Remarks.} For the simplest case \eqref{F_one}, it is
immediate that $f_1(n)=\Omega(\sqrt{n})$, hence also
$\r_{f_1}(x)=\Omega(\sqrt{x})$. The achievement of the elaborate
analysis to come is thus only an improvement by a logarithmic
factor. On the other hand, it is easy to see that, e.g.~for the
first example mentioned, it follows that $f^*(n)\ll\sqrt{n}(\log\log
n)^3$, hence our $\Omega$-bound cannot be deduced by consideration
of the individual values of the arithmetic function involved.

The situation may be compared with the sphere problem in $\RR^3$: If
$r_3(n)$ denotes the number of ways to write the positive integer
$n$ as a sum of three squares, then \be \notag \sum_{n\le x} r_3(n)
= \frac{4\pi}{3}\,x^{3/2}+\Omega\br{(x \log x)^{1/2}} \ee is the
best Omega-result known to date \cite{Szegoe}, while, by the very
same asymptotics, $r_3(n)=\Omega(\sqrt{n})$.

The method of proof which turned out to be appropriate in this
problem goes back to ideas due to Ramachandra \cite{Rama} and
Balasubramanian, Ramachandra \& Subbarao \cite{BalaRamaSub}. They
have been worked out in papers by Schinzel \cite{Schinzel},
K\"{u}hleitner \cite{MKtriples}, and the authors \cite{KNclass}.
However, in the present situation certain adaptions are necessary:
On the one hand, $f_H(n)$ is not as small as $O(n^\epsilon)$. On the
other hand, in the last step it will be advantageous to use special
properties of the Riemann zeta-function, instead of a general
theorem of Ramachandra's on Dirichlet \hbox{series \cite{Rama}.} It
should be mentioned that also in the case $M=0$, when the generating
function does not contain any factor involving $\zeta(2s)$, the
present approach seems to give a better result than Soundararajan's
method \cite{Sound} which up to date was most successful in the
divisor and circle problems.

\section{Preliminaries}
\subsection{} First of all, we can restrict our analysis to the case
that $H(s)=1$ identically, i.e., to \be\label{Ff} F(s) =
\sum_{n=1}^\infty \frac{f(n)}{n^s} =
\zeta^2(s)\zeta(2s-1)\zeta^M(2s)\qquad(\Re s>1)\,.\ee In fact,
assume that for some $f_H\in\mathfrak{C}$ and arbitrarily small
$c_0>0$, \be \abs{\r_{f_H}(x)}\le c_0\,\frac{\sqrt{x}(\log
x)^2}{(\log\log x)^{|M+1|}} \qquad(x\ge x_0)\,. \ee Let $$
\frac{1}{H(s)}=\sum_{n=1}^\infty h(n)n^{-s} \qquad \br{\Re
s>\frac{1}{3}}\,,$$ where the series converges absolutely for $\Re
s>\frac{1}{3}$. Since $f=f_H*h$, where $*$ denotes the convolution
of arithmetic functions, it readily follows that, for $x$
large\footnote{It is clear by the asymptotics with $O$-terms cited,
that the main term in this calculation amounts to
$\Res_{s=1}\br{F(s)\frac{x^s}{s}}$.}, \be \ba \sum_{n\le x} f(n) &=
\sum_{k\le x}h(k)\sum_{m\le x/k}f_H(m) = \sum_{k\le
x}h(k)\br{\frac{x}{k}\,p_{H}\br{\log\br{\frac{x}{k}}} +
\r_{f_H}\br{\frac{x}{k}}} \\ &=
\Res_{s=1}\br{F(s)\frac{x^s}{s}}+\sum_{k>x}h(k)\,\frac{x}{k}\,p_{H}\br{\log\br{\frac{x}{k}}}
+ \sum_{k\le x}h(k) \r_{f_H}\br{\frac{x}{k}}\,. \ea \ee From this it
is immediate that $$ \r_f(x) \ll \frac{c_0\,\sqrt{x} (\log
x)^2}{(\log\log x)^{|M+1|}}\,, $$ which yields a contradiction for
$c_0$ sufficiently small, provided that the Theorem has been
established for the case of \eqref{Ff}.

\subsection{} The assertion will be an easy consequence of the
following integral mean result.

\begin{prop} \label{prop_main} There exist
positive constants $B$ and $C_0$ with the property that \be
\label{prop} \Int_T^\infty \frac{|\r_{f}(u)|^2}{u^2}\,e^{-u/T^B}\,
\d u \ge \frac{C_0\,(\log T)^5}{(\log\log T)^{|2M+2|}} \ee for all
$T$ sufficiently large. \end{prop} The conclusion from this result
to Theorem 1 is easy and has many analogues in the literature; see
\cite{BalaRamaSub}, \cite{MKtriples}, \cite{KNclass}. Nevertheless,
we supply the details for convenience of the reader. Assume that for
any arbitrarily small constant $c_0>0$, there exists $u_0$ so that
\be \label{contra} \abs{\r_f(u)}\le \frac{c_0\,\sqrt{u} (\log
u)^2}{(\log\log u)^{|M+1|}} \qquad \hbox{for all}\quad u\ge u_0\,.
\ee Then, for $B$ as in the Proposition, and $T$ sufficiently large,
\be \notag \ba & \Int_T^\infty
\frac{|\r_{f}(u)|^2}{u^2}\,e^{-u/T^B}\, \d u \\ &\le c_0^2
\Int_T^\infty \inv{u}\, e^{-u/T^B} \frac{(\log u)^4}{(\log\log
u)^{|2M+2|}}\, \d u = c_0^2\br{\Int_T^{T^B}+ \Int_{T^B}^\infty} \\ &
\ll \frac{c_0^2\,(\log T)^4}{(\log\log
T)^{|2M+2|}}\Int_T^{T^B}\frac{\d u}{u} + c_0^2 \Int_1^\infty
\frac{e^{-v}(\log(T^B v))^4}{(\log\log(T^B v))^{|2M+2|}}\,\frac{\d
v}{v} \\ & \ll \frac{c_0^2\,(\log T)^5}{(\log\log T)^{|2M+2|}}\,.
\ea \ee For $c_0$ sufficiently small this contradicts \eqref{prop}.

\subsection{} One special feature of the present situation - as opposed
to the cases considered in \cite{BalaRamaSub}, \cite{Schinzel},
\cite{MKtriples}, \cite{KNclass} - is that the function $f(n)$ is
not "small", i.e., not $O(n^\epsilon)$ for each $\epsilon>0$.
However, if $f$ is generated by $\zeta^2(s) \zeta(2s-1)
(\zeta(2s))^M$, call $\hat f$ the arithmetic function generated by
$\zeta^2(s) \zeta(2s-1) (\zeta(2s))^{|M|}$. Then,
$$ |f(n)|\le \hat{f}(n) = \sum_{m_1m_2(m_3
k_1\dots k_{|M|})^2=n}\hskip-1cm{ m_3 } \ll n^{1/2+\epsilon}\,.
$$
Hence, using a trivial version of \eqref{Ass_fH}, applied to $\hat
f$, and summation by parts, it follows that \be \label{f_square}
\sum_{n\le X} n^\beta\,(f(n))^2 \ll X^{3/2+\beta+\epsilon} \ee for
each fixed $\beta>-\frac{3}{2}$, large $X$, and any $\epsilon>0$.

\subsection{} The following auxiliary result is classic and provides
some information that the factor of $F(s)$ involving $\zeta(2s)$ is
"not too harmful" close to the critical line.

\begin{lemma} \label{set_A} Let $\epsilon_0>0$ be a sufficiently small
constant. Then, for $\widehat{T}$ a sufficiently large real
parameter, there exists a set
$\A(\widehat{T})\subset[\widehat{T},2\widehat{T}]$ with the
following properties:\smallskip

\noindent{\rm(i)} $\A(\widehat{T})$ is the union of at most
$O(\widehat{T}^{\epsilon_0})$ open intervals, with a total length of
$O(\widehat{T}^{\epsilon_0})$.\smallskip

\noindent{\rm(ii)} $$
\sup_{t\in[\widehat{T},2\widehat{T}]\setminus\A(\widehat{T})}
\abs{\zeta(1+2it)}^{\pm1} \ll \log\log\widehat{T} $$

\noindent{\rm(iii)} There exist a real number $\delta(\epsilon_0)>0$
and a certain constant $C$ so that $$ \abs{\zeta(2s)}^{\pm1} \ll
\widehat{T}^C $$ uniformly in $\Re s\ge\h-\delta(\epsilon_0)$, $\Im
s\in[\widehat{T},2\widehat{T}]\setminus\A(\widehat{T})$.

\end{lemma}
\begin{proof} This result is contained in \cite[Theorem
1]{Rama} and \cite[Lemma 3.2]{RamaSanka}. An extension to Dedekind
zeta-functions, along with a neat proof, was given in \cite[Lemma
and formula (2.6)]{KNDioph}.
\end{proof}

\subsection{} We conclude this section by quoting a deep and
celebrated result due to Montgomery and Vaughan which provides a
mean-square bound for Dirichlet polynomials.

\begin{lemma} For an
arbitrary sequence of complex numbers $(\gamma_n)_{n=1}^\infty$ with
the property that $\sum_{n=1}^\infty n\,|\gamma_n|^2 $ converges,
and a large real parameter $X$,
$$   \Int_0^X \abs{\sum_{n=2}^\infty \gamma_n (n+u)^{-it} }^2 \d t =
\sum_{n=2}^\infty |\gamma_n|^2 \br{X + O(n)}\,,     $$ uniformly in
$-1\le u\le1$, and
$$   \Int_0^X \abs{\sum_{n=1}^\infty \gamma_n n^{-it} }^2 \d t =
 \sum_{n=1}^\infty |\gamma_n|^2 \br{X + O(n)}\,. $$ \end{lemma}
\begin{proof} This is an immediate consequence of Montgomery and Vaughan \cite[Corollary 2, formula
(1.9)]{MV}. \end{proof}

\section{Proof of Proposition \ref{prop_main}}

\subsection{} For positive real $T$ sufficiently large, we construct a
set $\U(T)$ on the real line as follows: Let $J:=\left[\frac{\log
T}{10\log2} \right]$, then $2^{-J}T\asymp T^{9/10}$. For
$j=1,\dots,J$, set $T_j:=2^{-j}T$. For some appropriate small
$\epsilon_0>0$, consider the sets $\A(T_j)$ furnished by Lemma
\ref{set_A}. Let $$ \A(T_j)=\bigcup_{i\in\I_j}]a_i^{(j)},b_i^{(j)}[
$$ be the decomposition of $\A(T_j)$ into
$\#(\I_j)=O(T^{\epsilon_0})$ open intervals of total length
$O(T^{\epsilon_0})$. Then we define \be \label{Def_UT} \U(T):=
\bigcup_{j=1}^J \bigcup_{i\in\I_j}]a_i^{(j)}-(\log
T)^2,b_i^{(j)}+(\log T)^2[\,. \ee By construction, it is clear that
$\U(T)$ consists of $O(T^{2\epsilon_0})$ open intervals of total
length $O(T^{2\epsilon_0})$.

\subsection{} We set $y=T^B$, with a suitably large
positive constant $B$, for throughout what follows. It suffices to
consider those values of $T$ for which \be\label{log_six}
 \Int_T^\infty \frac{|\r_f(u)|^2}{u^{2}}e^{-u/y}\d u \leq
(\log{T})^6\,. \ee (Otherwise the assertion of Proposition 1 is
obvious.) In this subsection, our aim is to deduce the asymptotic
representation \be \label{exp_series}  F(s) = \sum_{n=1}^\infty
\frac{f(n)}{n^s}\,e^{-n/y} \ +\ O(1)\,, \ee for \be\label{cond_s}
s=\h+it\,,\quad t\in[2^{-J}T,T]\setminus\U(T)\,, \ee which will be
assumed throughout the sequel.

By a version of Perron's formula (see, e.g.,
\cite[p.~380]{Prachar}), \be \label{Perron} \sum_{n=1}^\infty
\frac{f(n)}{n^s}e^{-n/y} = \frac{1}{2\pi i}
\Int_{2-i\infty}^{2+i\infty} F(s+w)y^w\Gamma(w) \,\d w\,.\ee We use
Stirling's formula in the crude form \be \Gamma(\sigma+i\tau) \asymp
\exp{\br{-\frac{\pi}{2}\vert \tau\vert}}\,\vert
\tau\vert^{\sigma-1/2}\,, \qquad (|\tau|\to\infty) \ee which holds
uniformly in every strip $\sigma_1\le\sigma\le\sigma_2$. From this
it is an immediate consequence that, for every fixed $k\in\ZZ_+$,
\be \label{Gamma_k} \Gamma^{(k)}(\sigma+i\tau) \ll
\exp{\br{-\frac{\pi}{4}\vert \tau\vert}}\, \qquad (|\tau|\to\infty)
\ee again uniformly in any strip of this kind.
                          \def\lt{(\log T)^2}
It is easy to see that we may break off the part of the integration
line in \eqref{Perron} corresponding to $|w|>(\log T)^2$, with an
error of only $O(1)$: \be \ba &\Int_{2\pm i(\log T)^2}^{2\pm
i\infty} F(s+w)\,y^w\,\Gamma(w)\,\d w  \\ & \ll
T^{2B}\Int_{\lt}^\infty w^{3/2}\,e^{-(\pi/2)w}\,\d w \ll T^{2B}
\,\exp\br{-{\pi\over4} \lt} \ \ll \ 1\,. \ea \ee Next, we replace
the remaining line of integration by a broken line segment $\frak L$
which joins (in this order) $2-i\lt$, $-\delta(\epsilon_0)-i\lt$,
$-\delta(\epsilon_0)+i\lt$, and $2+i\lt$, where $\delta(\epsilon_0)$
has the meaning as in Lemma \ref{set_A}. By Lemma \ref{set_A},
clause (iii), and known upper bounds for the zeta-function, \be
F(s+w)\ll T^{C'}\,, \ee with some positive constant $C'$, as long as
$w$ lies on $\frak L$ and $s$ satisfies \eqref{cond_s}. Therefore,
$$    \Int_{-\delta(\ep_0)-i\lt}^{-\delta(\ep_0)+i\lt} F(s+w) y^w \Gamma(w)
\,\d w \ll T^{C'-B\delta(\ep_0)} \Int_0^\infty
\Gamma(-\delta(\ep_0)+iu)\,\d u \ \ll\ 1 $$ if we have chosen $B\ge
C'/\delta(\ep_0)$, and
$$  \int_{-\delta(\ep_0)\pm i\lt}^{2\pm i\lt} F(s+w) y^w \Gamma(w) \,\d w \ll
T^{C'+2B}(\log T)^3 \exp\br{-{\pi\over2} \lt} \ \ll \ 1\,.  $$ Since
these integrals are bounded, the main contribution to the right-hand
side of \eqref{Perron} comes from the residue of the integrand
$F(s+w)y^w\Gamma(w)$ at $w=0$, which amounts to $F(s)$. This readily
yields \eqref{exp_series}.

\subsection{} It is an easy consequence of \eqref{log_six} that
there exists some $T^*\in[T,2T]$, not an integer, for which \be
\label{T_star_one} \frac{\abs{\r_f(T^*)}e^{-T^*/y}}{\sqrt{T^*}} \ll
(\log T)^3 \ee and \be \label{T_star_two} \inv{y}\Int_{T^*}^\infty
\frac{\abs{\r_f(u)}}{\sqrt{u}} e^{-u/y} \d u \ll (\log T)^3\,.\ee
This can be readily verified following the example of \cite[p.~111,
Lemma 4]{BalaRamaSub}. With this choice of $T^*$, we will split up
the series on the right-hand side of \eqref{exp_series}. In this
subsection, we shall handle \be\label{sum_n_large} \ba &
\sum_{n>T^*} \frac{f(n)}{n^s}e^{-n/y} = \Int_{T^*}^\infty u^{-s}
e^{-u/y}\,\d\br{\sum_{n\le u} f(n)} \\ & = \Int_{T^*}^\infty u^{-s}
e^{-u/y}\,\d\br{u\,p_H(\log u)} + \Int_{T^*}^\infty u^{-s} e^{-u/y}
\,\d\r_f(u) =: I_1 + I_2\,,\ea \ee where we have used Stieltjes
integrals. Integrating by parts, we obtain \be\label{bound_I_two}
\ba I_2 =& \left.\r_f(u) u^{-s} e^{-u/y} \right|_{u=T^*}^{u=\infty}
\\ & +s \Int_{T^*}^\infty u^{-s-1} e^{-u/y} \r_f(u)\d u +\inv{y}
\Int_{T^*}^\infty u^{-s} e^{-u/y} \r_f(u)\d u \\ =& s
\Int_\xi^{\xi+1}\sum_{n>T^*} (n+v)^{-s-1} e^{-(n+v)/y} \r_f(n+v) \d
v + O\br{(\log T)^3}\,. \ea \ee Here $ \xi:= T^*-[T^*]-1 $, and the
bounds \eqref{T_star_one}, \eqref{T_star_two} have been used.

To estimate $I_1$, we have to deal with integrals of the form $$
\Int_{T^*}^\infty u^{-s} (\log u)^r e^{-{u/y}}\,\d u\,, $$ with
$r\in\{0,1,2\}$, and $s$ satisfying \eqref{cond_s}. We write this as
\be \ba &\Int_{0}^\infty u^{-s} (\log u)^r e^{-{u/y}}\,\d u -
\Int_{0}^{T^*} u^{-s} (\log u)^r \,\d u \\ & + \Int_0^{T^*} u^{-s}
(\log u)^r\br{1-e^{-{u/y}}}\,\d u\ =:\ J_1 - J_2 +J_3\,. \ea \ee By
 Taylor expansion,
\be \notag  J_3 \ll {1\over y} \int_0^{T^*} u^{1/2} (\log u)^r\,\d u
\ll T^{-B+3/2+\epsilon} \ll 1\,. \ee Integrating by parts $r$ times,
we infer that $$ J_2 \ll \frac{T^{1/2} (\log T)^r }{ |1-s|} \ll
\frac{T^{1/2} (\log T)^r }{ T^{9/10}} \ll\ 1\,.  $$  Finally,
\be\notag \ba J_1 = &\ y^{1-s} \Int_0^\infty u^{-s}
(\log(uy))^r\,e^{-u}\,\d u
\\ = &\ y^{1-s} \sum_{\rho=0}^r {r\choose\rho} (\log y)^{r-\rho}\,
\Gamma^{(\rho)} (1-s)\ \ll\ 1\,, \ea \ee by \eqref{Gamma_k}.
Altogether this shows that $  I_1 \ll 1  $, hence, together with
\eqref{exp_series}, \eqref{sum_n_large} and \eqref{bound_I_two},
\be\label{Fs_mixed} \ba F(s)&= \sum_{n \leq T^*} \frac{f(n)}{n^s}e^{-n/y}\\
& +\ s\Int_\xi^{\xi+1} \sum_{n>T^*}
\frac{\r_f(n+u)}{(n+u)^{s+1}}e^{-(n+u)/y} \,\d u\ +\
O\br{(\log{T})^3} \ea \ee for $s$ satisfying \eqref{cond_s}.

\subsection{} The next step is to prove that \be \label{mean_square}
\Int_{[2^{-J}T,T]\setminus\U(T)} \abs{\frac{F(\h+it)}{\h+it}}^2 \d t
\ll 1 + \Int_T^\infty \frac{\abs{\r_f(u)}^2}{u^2}\,e^{-u/y}\,\d u\,.
\ee By \eqref{Fs_mixed}, the left-hand side of \eqref{mean_square}
is \be\label{RHS} \ba \ll &\ T^{-9/5} \Int_{2^{-J}T}^T \abs{\sum_{n
\leq T^*} \frac{f(n)}{n^{1/2+it}} e^{-n/y}}^2 \d t \\ & +
\Int_{2^{-J}T}^T \abs{\Int_\xi^{\xi+1} \sum_{n>T^*}
\frac{\r_f(n+u)}{(n+u)^{3/2+it}}e^{-(n+u)/y} \,\d u}^2 \d t\\ & +
O\br{T^{-9/10}(\log T)^6}\,. \ea \ee To bound the first integral
here, we use Lemma 2. In this way, \be\label{first}\ba & T^{-9/5}
\Int_{2^{-J}T}^T
\abs{\sum_{n \leq T^*} \frac{f(n)}{n^{1/2+it}} e^{-n/y}}^2 \d t\\
&\ll T^{-9/5}\sum_{n\le T^*}\frac{(f(n))^2}{n}\br{T+O(n)} \ll
T^{-3/10+\epsilon}\,,\ea \ee for any $\epsilon>0$, by an appeal to
\eqref{f_square}. Similarly, using Cauchy's inequality and Lemma 2
again, we see that the second term of \eqref{RHS} is \be\ba &\le
\Int_\xi^{\xi+1} \Int_0^T \abs{\sum_{n>T^*}
\frac{\r_f(n+u)}{(n+u)^{3/2+it}}e^{-(n+u)/y}}^2\,\d t\,\d u \\ & \ll
\Int_\xi^{\xi+1}\br{ \sum_{n>T^*}
\frac{\abs{\r_f(n+u)}^2}{(n+u)^{3}}e^{-2(n+u)/y}\,(T+O(n))}\,\d u
\\ & \ll
\Int_\xi^{\xi+1}\br{ \sum_{n>T^*}
\frac{\abs{\r_f(n+u)}^2}{(n+u)^{2}}e^{-(n+u)/y}}\,\d u \\ & \le
\Int_T^\infty \frac{\abs{\r_f(u)}^2}{u^2}\,e^{-u/y}\,\d u\,. \ea \ee
Together with \eqref{RHS} and \eqref{first}, this readily yields
\eqref{mean_square}.

\subsection{} In order to complete the proof of Proposition 1, it
remains to show that \be \label{mean_square_lower}
\Int_{[2^{-J}T,T]\setminus\U(T)} \abs{\frac{F(\h+it)}{\h+it}}^2 \d t
\gg \frac{(\log T)^5}{(\log\log T)^{|2M+2|}}\,. \ee In fact, by the
functional equation of the zeta-function (see \cite[p.~95]{TiZet}),
\be\notag \ba \abs{F\br{\h+it}} &=
\abs{\zeta^2\br{\h+it}\zeta(-2it)\br{\zeta\br{1+2it}}^M}\\ & \asymp
\abs{\zeta^2\br{\h+it}\br{\zeta\br{1+2it}}^{M+1}}|t|^{1/2}\ea\ee for
$|t|$ large. Further, by Lemma 1, clause (iii), and the construction
of $\U(T)$ in subsection 3.1, it follows that
$|\zeta(1+2it)|^{\pm1}\ll \log\log T$ for $t\in
[2^{-J}T,T]\setminus\U(T)$. Therefore, \be \label{LL_out}\ba
&\Int_{[2^{-J}T,T]\setminus\U(T)} \abs{\frac{F(\h+it)}{\h+it}}^2 \d
t \\ \gg & (\log\log T)^{-2|M+1|} \Int_{[2^{-J}T,T]\setminus\U(T)}
\abs{\zeta\br{\h+it}}^4\,\frac{\d t}{t}\,. \ea \ee For $t\in\U(T)$,
we use the classic pointwise bound $\zeta(\h+it)\ll|t|^{1/6+\ep}$:
See \cite[Theorem 5.5]{TiZet}. Since the total length of $\U(T)$ is
$O(T^{2\epsilon_0})$, it follows that \be \label{int_UT}
\Int_{\U(T)} \abs{\zeta\br{\h+it}}^4\,\frac{\d t}{t} \ll
T^{-7/30+3\ep_0}\,. \ee On the other hand, by the known asymptotics
for the fourth moment of the zeta-function (see
\cite[p.~129]{IvZet}), \be \label{fourth} \Int_{2^{-J}T}^T
\abs{\zeta\br{\h+it}}^4\,\frac{\d t}{t} \gg (\log T)^5\,. \ee
Combining \eqref{LL_out}, \eqref{int_UT}, and \eqref{fourth}, we
readily establish \eqref{mean_square_lower}. This completes the
proof of Proposition 1 and, by the observation in subsection 2.2,
also that of Theorem 1.


\end{document}